\documentclass[11pt]{article}

\newcommand\version{August 7, 2026}

\usepackage[T1]{fontenc}
\usepackage{lmodern}
\usepackage{amsmath,amssymb,amsthm,mathtools}
\usepackage{microtype}
\usepackage[letterpaper,margin=1in]{geometry}
\usepackage{enumitem}
\usepackage{hyperref}
\hypersetup{
  colorlinks=true,
  linkcolor=black,
  citecolor=black,
  urlcolor=black,
  pdftitle={A counterexample to the Kato conjecture for positive commutators},
  pdfauthor={Rupert L. Frank and Paata Ivanisvili}
}

\allowdisplaybreaks

\newtheorem{theorem}{Theorem}
\newtheorem{proposition}[theorem]{Proposition}
\newtheorem{lemma}[theorem]{Lemma}
\newtheorem{corollary}[theorem]{Corollary}
\newtheorem{conjecture}[theorem]{Conjecture}
\theoremstyle{definition}

\theoremstyle{remark}
\newtheorem{remark}[theorem]{Remark}

\newcommand{\R}{\mathbb{R}}
\newcommand{\C}{\mathbb{C}}

\newcommand{\Fock}{\mathcal{F}}

\newcommand{\ip}[2]{\left\langle #1,#2\right\rangle}
\newcommand{\norm}[1]{\left\lVert #1\right\rVert}
\newcommand{\abs}[1]{\left\lvert #1\right\rvert}
\newcommand{\dd}{\,\mathrm{d}}

\newcommand{\Tr}{\operatorname{Tr}}

\title{A counterexample to the Kato conjecture\\for positive commutators}
\author{Rupert L. Frank\thanks{Mathematisches Institut, Ludwig-Maximilians Universit\"at M\"unchen, Theresienstr.~39, 80333 M\"unchen, Germany; Munich Center for Quantum Science and Technology, Schellingstr.~4, 80799 M\"unchen, Germany. Email: \texttt{r.frank@lmu.de}.}
\qquad
Paata Ivanisvili\thanks{Department of Mathematics, University of California, Irvine, 510C Rowland Hall, Irvine, CA 92697-3875, USA. Email: \texttt{pivanisv@uci.edu}.}}
\date{\version}

\begin{document}
\maketitle

\begin{abstract}
We disprove the conjectural converse to Kato's positivity criterion for commutators of functions of the canonical position and momentum operators $Q$ and $P$ by showing that the operator
\[
i\,[\,\arctan(P),\,\arctan(Q)\,]
\]
is nonnegative and nonzero.
\end{abstract}




\section{Introduction and main results}

Let $P$ and $Q$ be the unbounded, selfadjoint operators in $L^2(\R)$ given by $P=-i\frac{d}{dx}$ and $Q=x$, meaning multiplication by the coordinate $x$, with their natural domains. Among other things, we shall prove the following result.

\begin{theorem}\label{thm:main}
The operator
\begin{equation}\label{eq:main-positive}
C:=i\,[\,\arctan(P),\,\arctan(Q)\,]
\end{equation}
is nonnegative. It is trace class with
\begin{equation}\label{eq:trace-main}
\operatorname{Tr} C=\frac{\pi}{2} \,.
\end{equation}
\end{theorem}

One reason this theorem is interesting is because it gives a counterexample to a conjecture by Kato, as we explain next.


\subsection*{The Howland--Kato commutator problem}

If $f$ and $g$ are bounded real-valued Borel functions, then $f(P)$ and $g(Q)$ are bounded self-adjoint operators, and therefore
\[
i[\,f(P),\, g(Q)\,]:=i\,\bigl(f(P)g(Q)-g(Q)f(P)\bigr)
\]
is bounded and self-adjoint. 

For $a>0$, put
\[
S_a:=\{z\in\C:\abs{\operatorname{Im}z}<a\}
\]
and let $K_a$ denote the class of bounded real-valued functions $h$ on $\R$ that admit an analytic continuation, still denoted by $h$, to $S_a$ and satisfy
\begin{equation}\label{eq:Kato-sign}
\operatorname{Im}h(z)\,\operatorname{Im}z\geq 0
\qquad \text{for all}\ z\in S_a \,.
\end{equation}
Kato \cite{Kato} proved that
\begin{equation}\label{eq:Kato-sufficient}
f\in K_a,\quad g\in K_b,\quad ab\geq \frac{\pi}{2}
\qquad\Longrightarrow\qquad
i[\,f(P),g(Q)\,]\geq 0 \,.
\end{equation}
Since restriction to a smaller strip preserves membership in the corresponding Kato class, the condition $ab\geq\pi/2$ can equivalently be written with equality after decreasing one of the strip widths.

The conjectural converse, in the formulation of Herbst and Kriete \cite{HerbstKriete}, is the following.

\begin{conjecture}[Kato conjecture]\label{conj:Kato}
Let $f$ and $g$ be bounded real-valued functions and suppose that
\[
i[\,f(P),\,g(Q)\,]\geq 0,
\qquad
i[\,f(P),\, g(Q)\,]\neq 0.
\]
After replacing both $f$ and $g$ by their negatives if necessary, there are $a,b>0$ with $ab=\pi/2$ such that
\[
f\in K_a,
\qquad
g\in K_b.
\]
\end{conjecture}

More precisely, of course, the conjecture is that there are function $\tilde f\in K_a$ and $\tilde g\in K_b$, coinciding with $f$ and $g$ almost everywhere.

It is easy to see that the function $\arctan$ belongs to $K_1$, but not to $K_a$ for any $a>1$, and that the function $-\arctan$ does not belong to $K_a$ for any $a>0$; see Proposition \ref{prop:arctan-strip}. Therefore, Theorem~\ref{thm:main} implies the following.

\begin{corollary}
Conjecture \ref{conj:Kato} fails.
\end{corollary}

Let us put Conjecture \ref{conj:Kato} into its historical context. The problem arose from work of Howland \cite{Howland} on dense point spectrum.  In 1987 he observed that the pair
\[
f(x)=\arctan(x/2),
\qquad
g(x)=\tanh x
\]
gives a nonnegative commutator $i[\,f(P),g(Q)\,]$.  Kato subsequently began a systematic study of all pairs of bounded real functions producing a positive commutator \cite{Kato} and proved several interesting results.

An observation in Kato's paper is that membership in $K_a$ is equivalent to a positive-measure representation by translates of the hyperbolic tangent.  More precisely, if $\widehat a=\pi/(2a)$, then, up to an additive real constant, elements of $K_a$ have the form
\begin{equation}\label{eq:Kato-representation}
h(x)=\int_{\R}\tanh\bigl(\widehat a(x-t)\bigr)\dd\mu(t),
\end{equation}
where $\mu$ is a finite nonnegative measure.  This representation, together with an explicit study of the commutators $i[\,\tanh\alpha P,\tanh\beta Q\,]$, leads to the sufficient condition \eqref{eq:Kato-sufficient}.  Kato also solved the rank-one case and wrote that there was reason to expect the sufficient construction to describe all nonzero positive commutators.

Herbst and Kriete \cite{HerbstKriete} renewed the interest in the Howland--Kato commutator problem and Kato's conjecture. They proved several structural facts and formulated Conjecture~\ref{conj:Kato} explicitly.  Among other results, they showed that if a nonzero commutator is positive, then suitable versions of $f$ and $g$ are monotone and continuous. Froese and Herbst later established the absolute continuity of these versions and the integrability of their derivatives \cite[Theorem~1.1]{FroeseHerbst}.  The finite-rank problem was studied further in \cite{HerbstFiniteRank}.

Froese and Herbst \cite{FroeseHerbst} also obtained additional evidence for the conjecture and connected the problem with Loewner's theorem on operator-monotone functions.  In particular, they proved the conjectured strip conclusion under one-sided exponential-moment assumptions on the derivatives.  The example in Theorem~\ref{thm:main} lies outside that regime, since
\[
f'(x)=g'(x)=\frac{1}{1+x^2}
\]
has no nontrivial exponential moment in either direction.


\subsection*{A generalization and outline of the proof}

Theorem \ref{thm:main} is a special case of a more general construction, which we now outline. Let
\begin{equation}\label{eq:F-def}
F(t):=\ln\frac{t}{\sin t}
\qquad\text{for}\ t\in(-\pi,\pi) \,,
\end{equation}
where we note that the ratio $t/\sin t$ is positive on this interval and where we interpret this ratio as $1$ at $t=0$. This defines an even smooth function with
\begin{equation}\label{eq:Fpp}
F''(t)=\frac{1}{\sin^2t}-\frac{1}{t^2},
\qquad
F''(0)=\frac13.
\end{equation}

Let
\[
I:=\left(-\frac{\pi}{2},\frac{\pi}{2}\right)
\]
and let $T$ be the integral operator in $L^2(I)$ with integral kernel
\begin{equation}\label{eq:k-def}
T(\theta,\varphi)
:=\cos\theta\cos\varphi\,F''(\theta-\varphi),
\qquad \theta,\varphi\in I \,.
\end{equation}
The kernel in \eqref{eq:k-def} is bounded (see also the proof of Lemma \ref{lem:k-bound}), so $T$ is bounded (and even compact).

The operator $T$ leaves the subspaces of even and odd functions on $I$ invariant and we let $T_{\rm even}$ and $T_{\rm odd}$ denote the parts of $T$ on these subspaces.

\begin{proposition}\label{prop:poscommgen}
If $\alpha,\beta\in\R$ satisfy $0<\alpha\beta\leq 2\|T_{\rm even}\|^{-1}$, then
\[
i\,[\,\arctan(\alpha P),\arctan(\beta Q)\,]\geq0.
\]
Moreover, this operator is trace class and has trace
$\pi/2$.
\end{proposition}

\begin{proposition}\label{prop:normbound}
    We have $\|T_{\rm even}\| \leq 0.64$.
\end{proposition}

This bound is rather tight, as from Lemma \ref{lem:trace-and-rayleigh} below it follows that $\|T_{\rm even}\|\geq 0.62$. Without proof we mention that numerical computations (both a Galerkin computation and a Gauss--Legendre Nystr\"om discretization) suggest $\|T_{\rm even}\|\approx 0.6368$.

Note that Theorem \ref{thm:main} is an immediate consequence of Propositions \ref{prop:poscommgen} and \ref{prop:normbound}. In fact, the proof of Theorem \ref{thm:main} is much simpler. Instead of Proposition \ref{prop:poscommgen}, we only need the weaker version with $\|T_{\rm even}\|$ replaced by $\|T\|$, which makes Lemma \ref{lem:T-parity} unnecessary. Moreover, instead of Proposition \ref{prop:normbound} we only need the much simpler bound $\|T\|\leq 2$, which is proved in Lemma \ref{lem:k-bound}.


\subsection*{Outline of this paper}

In Sections \ref{sec:poscommgen} and \ref{sec:normbound} we will prove Propositions \ref{prop:poscommgen} and \ref{prop:normbound}, respectively. In the final Section \ref{sec:arctank1} we determine the maximal Kato strip of the arctangent.


\section{Proof of Proposition \ref{prop:poscommgen}}\label{sec:poscommgen}

Most of our work in this section concerns the family of operators
$C_s$, $s>0$, defined by
\[
C_s:=i\,[\,\arctan(P/s),\arctan(Q/s)\,]
\qquad\text{in }L^2(\R).
\]


\subsection*{The integral kernel of $C_s$}

\begin{lemma}\label{lem:intkernel}
The operator $C_s$ has integral kernel
\[
C_s(x,y)=\frac12\,
\frac{s}{\sqrt{s^2+x^2}\sqrt{s^2+y^2}}\,e^{-D_s(x,y)},
\qquad x,y\in\R,
\]
where
\begin{equation}\label{eq:Ds-def}
D_s(x,y):=s\abs{x-y}
-F\bigl(\vartheta_s(x)-\vartheta_s(y)\bigr)
\end{equation}
and
\begin{equation}\label{eq:theta-s}
\vartheta_s(x):=\arctan(x/s).
\end{equation}
\end{lemma}

\begin{proof}
Using the convention
\begin{equation}\label{eq:Fourier-convention}
\widehat h(\xi)=\frac{1}{\sqrt{2\pi}}
\int_{\R}e^{-ix\xi}h(x)\dd x
\end{equation}
for the Fourier transform, the integral kernel of
$i[\,f(P),g(Q)\,]$ is
\begin{equation}\label{eq:general-kernel}
\frac{1}{\sqrt{2\pi}}\,
\frac{g(x)-g(y)}{x-y}\,\widehat{f'}(y-x).
\end{equation}
Under the assumptions that $f$ is monotone increasing and bounded and
that $g$ is smooth and bounded with bounded derivatives, this is shown
in \cite[Lemma~3.3]{HerbstKriete}.

Since
\[
\vartheta_s'(x)=\frac{s}{s^2+x^2},
\qquad
\widehat{\vartheta_s'}(\xi)=\sqrt{\frac\pi2}\,e^{-s|\xi|},
\]
we obtain
\begin{equation}\label{eq:Cs-kernel}
C_s(x,y)=\frac12e^{-s\abs{x-y}}
\frac{\vartheta_s(x)-\vartheta_s(y)}{x-y}.
\end{equation}
Writing $x=s\tan\theta$ and $y=s\tan\varphi$, we have
\begin{equation}\label{eq:tan-difference}
x-y=s\bigl(\tan\theta-\tan\varphi\bigr)
=s\,\frac{\sin(\theta-\varphi)}{\cos\theta\cos\varphi},
\end{equation}
so
\begin{align*}
C_s(x,y)
&=\frac{1}{2s}\cos\theta\cos\varphi\,e^{-s\abs{x-y}}
\frac{\theta-\varphi}{\sin(\theta-\varphi)}\\
&=\frac{\cos\theta\cos\varphi}{2s}
\exp\!\left(-s\abs{x-y}+F(\theta-\varphi)\right)\\
&=\frac12\,
\frac{s}{\sqrt{s^2+x^2}\sqrt{s^2+y^2}}\,e^{-D_s(x,y)}.
\end{align*}
This completes the proof of the lemma.
\end{proof}

\begin{remark}
The divided-difference kernel
\begin{equation}\label{eq:kernelfree}
\frac{\vartheta_s(x)-\vartheta_s(y)}{x-y},
\qquad x,y\in\R,
\end{equation}
is not positive semidefinite for any $s>0$, so nonnegativity of
$C_s$ emerges only after multiplication by $e^{-s\abs{x-y}}$.  Indeed,
positive semidefiniteness of \eqref{eq:kernelfree} on all of $\R$ would
be equivalent to operator monotonicity of $\vartheta_s$ on the whole
real line \cite[Theorem~5.1]{Simon2}.  This is
impossible for a nonconstant bounded function
\cite[Theorem~1.1]{Simon2}.
\end{remark}


\subsection*{A formula for $D_s(x,y)$}

Next, we derive a convenient representation of $D_s(x,y)$. For $x\in\R$, we define $h_x\in L^2(\R)$ by
\begin{equation}\label{eq:hx-def}
h_x(t):=
\begin{cases}
\mathbf 1_{(0,x)}(t)& \text{if}\ x\geq0,\\
-\mathbf 1_{(x,0)}(t)& \text{if}\ x<0.
\end{cases}
\end{equation}
Also, let $B_s$ be the integral operator in $L^2(\R)$ with integral kernel
\begin{equation}\label{eq:Bs-kernel}
B_s(x,y)=\vartheta_s'(x)\vartheta_s'(y)
F''\bigl(\vartheta_s(x)-\vartheta_s(y)\bigr).
\end{equation}

\begin{lemma}\label{lem:drep}
For all $x,y\in\R$,
\begin{equation}\label{eq:drep}
D_s(x,y)=\frac12\ip{h_x-h_y}{(2sI-B_s)(h_x-h_y)}
\qquad\text{for all}\ x,y\in\R \,.
\end{equation}
\end{lemma}

\begin{proof}
By symmetry we may assume $x<y$.  Then
$h_y-h_x=\mathbf 1_{(x,y)}$ and, from \eqref{eq:Bs-kernel},
\begin{align}
\ip{\mathbf 1_{(x,y)}}{(2sI-B_s)\mathbf 1_{(x,y)}}
&=2s(y-x)-\int_x^y\int_x^y
\vartheta_s'(r)\vartheta_s'(t)
F''\bigl(\vartheta_s(r)-\vartheta_s(t)\bigr)
\dd r\dd t.
\label{eq:interval-step1}
\end{align}
Put $\alpha=\vartheta_s(x)$ and $\beta=\vartheta_s(y)$.  The change
of variables $u=\vartheta_s(r)$ and $v=\vartheta_s(t)$ turns the
double integral into
\begin{equation}\label{eq:double-Fpp}
\int_\alpha^\beta\int_\alpha^\beta F''(u-v)\dd u\dd v.
\end{equation}
Writing $L=\beta-\alpha>0$, translation invariance and evenness give
\[
\int_\alpha^\beta\int_\alpha^\beta F''(u-v)\dd u\dd v
=2\int_0^L(L-r)F''(r)\dd r=2F(L),
\]
where the last identity follows by integration by parts and
$F(0)=F'(0)=0$.  Substitution into \eqref{eq:interval-step1} yields
\[
\ip{\mathbf 1_{(x,y)}}{(2sI-B_s)\mathbf 1_{(x,y)}}
=2s(y-x)-2F(\beta-\alpha)=2D_s(x,y),
\]
which is \eqref{eq:drep}.
\end{proof}


\subsection*{The signs of the parity parts of $T$}

\begin{lemma}\label{lem:T-parity}
The parity restrictions of $T$ satisfy
\[
T_{\rm even}\geq0,
\qquad
T_{\rm odd}\leq0.
\]
\end{lemma}

\begin{proof}
For $|t|<\pi$ one has
\begin{equation}\label{eq:fdoubleprimerepr}
F''(t)=2\int_0^\infty
\frac{r\cosh(rt)}{e^{\pi r}-1}\dd r.
\end{equation}
Indeed, this follows from the facts that
\begin{align*}
F''(t)
&=\frac1{\sin^2t}-\frac1{t^2}
=\sum_{j=1}^\infty
\left(\frac1{(j\pi-t)^2}+\frac1{(j\pi+t)^2}\right)
=2\sum_{j=1}^\infty
\frac{(j\pi)^2+t^2}{\bigl((j\pi)^2-t^2\bigr)^2},
\end{align*}
while
\[
\int_0^\infty re^{-j\pi r}\cosh(rt)\dd r
=\frac{(j\pi)^2+t^2}{\bigl((j\pi)^2-t^2\bigr)^2}.
\]

For $v\in L^2(I)$, Fubini's theorem and
\[
\cosh(r(\theta-\varphi))
=\cosh(r\theta)\cosh(r\varphi)
-\sinh(r\theta)\sinh(r\varphi)
\]
give
\begin{align*}
\ip{v}{Tv}
={}&2\int_0^\infty\frac{r}{e^{\pi r}-1}
\left|\int_Iv(\theta)\cos\theta\cosh(r\theta)\dd\theta\right|^2\dd r
-2\int_0^\infty\frac{r}{e^{\pi r}-1}
\left|\int_Iv(\theta)\cos\theta\sinh(r\theta)\dd\theta\right|^2\dd r.
\end{align*}
Since $\cos\theta \cosh(r\theta)$ is an even function of $\theta$, the first term only sees the even part of $v$ and, similarly, the second part only sees the odd part of $v$. Thus, for $v$ from the relevant subspaces,
    \begin{align*}
        \ip{v}{T_{\rm even} v} & = 2 \int_0^\infty \frac{r}{e^{\pi r}-1} \left| \int_I v(\theta) \cos\theta \cosh(r\theta) \dd\theta\right|^2 \dd r \,,\\
        \ip{v}{T_{\rm odd} v} & = - 2 \int_0^\infty \frac{r}{e^{\pi r}-1} \left| \int_I v(\theta) \cos\theta \sinh(r\theta) \dd\theta\right|^2 \dd r \,.
    \end{align*}
    These formulas imply, in particular, the claimed assertion about the signs of the operators. 
\end{proof}


\subsection*{Proof of Proposition \ref{prop:poscommgen}}

We divide the proof of Proposition \ref{prop:poscommgen} into several steps.

\medskip

\emph{Step 0.} We being with some simple reductions. Since $\arctan$ is odd, we may assume that $\alpha,\beta$ are both positive. Next, by a simple scaling argument we may assume that $\alpha=\beta$ and we call this common parameter $1/s$. Thus, it suffices to consider the operator $C_s$.

\medskip

\emph{Step 1.} Combining Lemmas \ref{lem:intkernel} and \ref{lem:drep}, we see that the integral kernel of $C_s$ is given by
    \begin{equation}
        \label{eq:posommgenproof}
            C_s(x,y) = \frac12 \, \frac{s}{\sqrt{s^2+x^2}\sqrt{s^2+y^2}} \, e^{-\frac12 \langle (h_x-h_y),A_s(h_x-h_y)\rangle} \,,
    \end{equation}
    where we abbreviated $A_s:=2s-B_s$. (This is short for $A_s=2sI-B_s$ with $I$ the identity.)

\medskip

\emph{Step 2: Nonnegativity of $A_s$.}
    We begin by noting that the operator $B_s$ is unitarily equivalent to the operator $s^{-1} T$. To see this, define $U_s:L^2(\R)\to L^2(I)$ by
\begin{equation}\label{eq:Us-def}
(U_su)(\theta):=\sqrt{s}\,\sec\theta\,u(s\tan\theta).
\end{equation}
Since $\dd x=s\sec^2\theta\dd\theta$, this map is unitary.  If $v=U_su$, then
\[
u(s\tan\theta)=\frac{1}{\sqrt{s}}v(\theta)\cos\theta
\]
and
\[
\vartheta_s'(x)\dd x=\dd\theta.
\]
Therefore
\begin{align*}
    \ip{u}{B_su} & = \iint_{\R\times\R}\overline{u(x)}u(y)
\vartheta_s'(x)\vartheta_s'(y)
F''\bigl(\vartheta_s(x)-\vartheta_s(y)\bigr)
\dd x\dd y \\
& = \frac1s \iint_{I\times I} \overline{v(\theta)}v(\varphi) \cos\theta\cos\varphi\,F''(\theta-\varphi) \dd\theta \dd\varphi
= \frac1s\ip{v}{Tv}_{L^2(I)}.
\end{align*}
By polarization, this gives the claimed unitary equivalence.

As a consequence, $A_s$ is unitarily equivalent to $2s - s^{-1}T$, which is equal to the direct sum of $2s-s^{-1}T_{\rm even}$ and $2s-s^{-1} T_{\rm odd}$ on the even and odd subspaces of $L^2(I)$. By Lemma \ref{lem:T-parity}, the latter operator is nonnegative. The assumption $s^{-2} = \alpha\beta\leq 2\|T_{\rm even}\|^{-1}$ implies that $2s-s^{-1}T_{\rm even}$ is nonnegative, thus proving the nonnegativity of $A_s$.

\medskip

\emph{Step 3: Fock-space factorization and the trace.}
Let $\mathcal H:=L^2(\R)$ and put
\[
\Phi_x:=A_s^{1/2}h_x,
\qquad
q_x:=\norm{\Phi_x}^2.
\]
The operator $A_s$ commutes with complex conjugation, so each
$\Phi_x$ is real-valued and $\ip{\Phi_x}{\Phi_y}$ is real.  By
Lemma~\ref{lem:drep},
\begin{equation}\label{eq:Ds-distance}
D_s(x,y)=\frac12\norm{\Phi_x-\Phi_y}^2.
\end{equation}
Consider the symmetric Fock space
\[
\Fock_{\rm s}(\mathcal H):=\bigoplus_{k=0}^\infty \mathcal H^{\otimes_{\rm s}k},
\qquad \mathcal H^{\otimes_{\rm s}0}:=\C,
\]
and, for $\Phi\in \mathcal H$, its exponential vector
\[
\operatorname{Exp}(\Phi)
:=\bigoplus_{k=0}^\infty\frac{\Phi^{\otimes_{\rm s}k}}{\sqrt{k!}}.
\]
The identities
\begin{equation}\label{eq:exp-vector-identities}
\ip{\operatorname{Exp}(\Phi)}{\operatorname{Exp}(\Psi)}
=e^{\ip{\Phi}{\Psi}},
\qquad
\norm{\operatorname{Exp}(\Phi)}^2=e^{\norm{\Phi}^2}
\end{equation}
hold by direct expansion of the series. Defining
\begin{equation}\label{eq:Psi-x}
\Psi_x:=\sqrt{\frac{s}{2}}\,
\frac{e^{-q_x/2}}{\sqrt{s^2+x^2}}\,
\operatorname{Exp}(\Phi_x)
\in\Fock_{\rm s}(\mathcal H) \,,
\end{equation}
equations \eqref{eq:Ds-distance},
\eqref{eq:exp-vector-identities} and \eqref{eq:posommgenproof} imply
\begin{align}
\ip{\Psi_x}{\Psi_y}
&=\frac{s}{2\sqrt{s^2+x^2}\sqrt{s^2+y^2}} \,
\exp\!\left(-\frac{q_x+q_y}{2}+\ip{\Phi_x}{\Phi_y}\right)\notag\\
&=\frac{s}{2\sqrt{s^2+x^2}\sqrt{s^2+y^2}} \,e^{-D_s(x,y)}
=C_s(x,y).
\label{eq:Cs-Gram}
\end{align}

Define $V:L^2(\R)\to\Fock_{\rm s}(\mathcal H)$ by
\[
Vu:=\int_{\R}\Psi_xu(x)\dd x \,,
\]
where the definition is understood in the weak sense. This is well-defined, since $x\mapsto\Psi_x$ is weakly measurable (since $\norm{h_x-h_y}^2=\abs{x-y}$ implies that $x\mapsto h_x$ is continuous as an $L^2$-valued map) and since, from
\eqref{eq:exp-vector-identities},
\begin{equation}\label{eq:Psi-norm}
\norm{\Psi_x}^2
=\frac{s}{2(s^2+x^2)},
\end{equation}
which implies
\begin{equation}\label{eq:Psi-L2}
\int_{\R}\norm{\Psi_x}^2\dd x
=\frac12\int_{\R}\frac{s}{s^2+x^2}\dd x
=\frac\pi2.
\end{equation}

It follows from \eqref{eq:Cs-Gram} that
\[
C_s=V^*V\geq0,
\]
and, using the standard criterion to verify the Hilbert--Schmidt property,
\[
\operatorname{Tr}C_s=\norm{V}_{\rm HS}^2=\int_{\R}\norm{\Psi_x}^2\dd x=\frac\pi2.
\]
This completes the proof of Proposition \ref{prop:poscommgen}.
\qed

\begin{remark}
    The argument in Step 3 is of a general nature and related to work of Schoenberg \cite{Schoenberg}. Under the hypothesis $A_s\geq0$, Lemma~\ref{lem:drep} expresses
$D_s$ as one half of a squared Hilbert-space distance.  In
Schoenberg's terminology, $D_s$ is therefore conditionally negative
definite, and Schoenberg's theorem \cite{Schoenberg} implies directly
that $e^{-D_s}$ is a positive-semidefinite kernel.  The Fock-space
argument above is an explicit factorization of this conclusion and,
in addition, yields the trace-class assertion.
\end{remark}


\section{Proof of Proposition \ref{prop:normbound}}\label{sec:normbound}

\subsection{A simple bound}

We emphasize again that the following simple bound is enough to prove Theorem \ref{thm:main}.

\begin{lemma}\label{lem:k-bound}
    We have $\|T_{\rm even}\| \leq \|T\| \leq \frac{3\pi}5$.
\end{lemma}

\begin{proof}
    \emph{Step 1.} In this step we shall show that, for all $\theta,\varphi\in I$,
    \begin{equation}\label{eq:k-bound}
        0\leq T(\theta,\varphi)\leq\frac35.
    \end{equation}

Put $t:=\theta-\varphi$.  Since $\abs{\sin t}\leq\abs t$ for $\abs t<\pi$, formula \eqref{eq:Fpp} gives $F''(t)\geq0$.  Also $\cos\theta$ and $\cos\varphi$ are positive on $I$, proving the lower bound. We prove the upper bound in two regions.

\smallskip
\noindent\emph{Case 1: $\abs t\geq\pi/2$.}
We have
\begin{equation}\label{eq:cos-product-bound}
\cos\theta\cos\varphi = \frac{\cos(\theta-\varphi)+\cos(\theta+\varphi)}{2}
\leq \frac{1+\cos t}{2}
=\cos^2\frac t2.
\end{equation}
Dropping the negative term $-t^{-2}$ in \eqref{eq:Fpp}, we obtain
\begin{align*}
T(\theta,\varphi)
\leq \cos^2\frac t2\,\frac{1}{\sin^2t}
=\frac{1}{4\sin^2(t/2)}
\leq\frac12.
\end{align*}

\smallskip
\noindent\emph{Case 2: $\abs t\leq\pi/2$.}
We claim that $F''$ is increasing on $[0,\pi/2]$.  Indeed,
\begin{equation}\label{eq:F-third}
F'''(t)=\frac{2}{t^3}-\frac{2\cos t}{\sin^3t}.
\end{equation}
Thus $F'''(t)\geq0$ is equivalent to
\begin{equation}\label{eq:sine-cosine-cubic}
\left(\frac{\sin t}{t}\right)^3\geq\cos t.
\end{equation}
To prove this, set
\[
h(t):=3\ln\frac{\sin t}{t}-\ln\cos t,
\qquad 0<t<\frac{\pi}{2}.
\]
Then
\begin{equation}\label{eq:h-prime}
h'(t)=3\cot t-\frac3t+\tan t
=\frac{r(t)}{t\sin t\cos t},
\end{equation}
where
\[
r(t):=t(1+2\cos^2t)-3\sin t\cos t.
\]
A direct differentiation gives
\begin{equation}\label{eq:r-prime}
r'(t)=4\sin t\bigl(\sin t-t\cos t\bigr).
\end{equation}
The function $\sin t-t\cos t$ vanishes at zero and has derivative $t\sin t>0$.  Hence it is positive for $t>0$.  It follows from \eqref{eq:r-prime} that $r'(t)>0$, and since $r(0)=0$, we have $r(t)>0$.  Equation \eqref{eq:h-prime} now gives $h'(t)>0$.  Since $h(t)\to0$ as $t\downarrow0$, we obtain $h(t)\geq0$, which is exactly \eqref{eq:sine-cosine-cubic}.  This proves the claim that $F''$ is increasing.

Therefore, in the present region,
\[
F''(t)\leq F''\left(\frac{\pi}{2}\right)
=1-\frac{4}{\pi^2}<\frac35.
\]
Since $\cos\theta\cos\varphi\leq1$, the required bound follows.

\medskip

\emph{Step 2.} The kernel $T(\theta,\varphi)$ is symmetric and, by Step 1,
\[
\sup_{\theta\in I}\int_I\abs{T(\theta,\varphi)}\dd\varphi
\leq \frac35\abs I=\frac{3\pi}{5}.
\]
The lemma now follows from Schur's test.
\end{proof}


\subsection{An almost optimal bound}

In the remainder of this section, we prove the full strength of Proposition \ref{prop:normbound}. This is an entertaining exercise, but as we have already mentioned after stating this proposition, it can be ignored by readers who are only interested in the proof of Theorem \ref{thm:main}.


\subsection*{A trace--residual estimate}

Clearly, for a nonnegative trace class operator $A$, we have $\|A\|\leq\Tr A$. For our operator $T_{\rm even}$, which is nonnegative by Lemma \ref{lem:T-parity}, this, together with an evaluation of the trace in \eqref{eq:tau-exact} below, will give the bound $\|T_{\rm even}\|\leq 0.725$, which is already a significant improvement over Lemma \ref{lem:k-bound}. To further improve the bound, we will improve the abstract bound $\|A\|\leq\Tr A$ under the assumption that we have an `approximate eigenvector'.

\begin{lemma}\label{lem:trace-residual}
Let $A$ be a nonnegative trace-class operator on a Hilbert space, let
$v$ be a unit vector, and set
\[
a:=\ip{v}{Av},
\qquad
\rho:=\norm{(I-v\otimes v)Av},
\qquad
\tau:=\operatorname{Tr}A.
\]
Then
\begin{equation}\label{eq:trace-residual}
\norm{A}\leq
\frac{\tau+\sqrt{(2a-\tau)^2+4\rho^2}}{2}.
\end{equation}
\end{lemma}

\begin{proof}
With respect to the orthogonal decomposition
$\mathbb C v\oplus v^\perp$, write
\[
A=\begin{pmatrix}a&b^*\\ b&C\end{pmatrix}.
\]
Then $C\geq0$, $\norm b=\rho$, and
\[
\norm C\leq\operatorname{Tr}C=\tau-a.
\]
Consequently, for $z\in\mathbb C$ and $w\in v^\perp$,
\[
\ip{zv+w}{A(zv+w)}
\leq a\abs z^2+2\rho\abs z\norm w+(\tau-a)\norm w^2.
\]
The right side is the quadratic form of the two-by-two matrix
\[
\begin{pmatrix}a&\rho\\ \rho&\tau-a\end{pmatrix}.
\]
Its largest eigenvalue is the right side of
\eqref{eq:trace-residual}, and the assertion follows.
\end{proof}


\subsection*{The trace and a Rayleigh quotient}

We will apply Lemma \ref{lem:trace-residual} with $A=T_{\rm even}$ and with $v$ given by
\begin{equation}\label{eq:e0-def}
e_0(\theta):=\sqrt{\frac2\pi}\cos\theta \,,
\end{equation}
which is normalized in $L^2(I)$. As usual, we write
\[
\operatorname{Si}(x):=\int_0^x\frac{\sin t}{t}\dd t,
\qquad
\operatorname{Cin}(x):=\int_0^x\frac{1-\cos t}{t}\dd t.
\]

\begin{lemma}\label{lem:trace-and-rayleigh}
We have
\begin{align}
\tau&:= \Tr T_{\rm even} = \frac{\pi}{12}+\frac14\operatorname{Si}(\pi),
\label{eq:tau-exact}\\
a_0& :=\ip{e_0}{T_{\rm even}e_0}
=\operatorname{Si}(2\pi)
+\frac{\operatorname{Cin}(2\pi)}{\pi}-\frac\pi2.
\label{eq:a0-exact}
\end{align}
In particular,
\begin{equation}\label{eq:tau-a0-numerical}
\tau<0.725,
\qquad
0.6232<a_0<0.6234.
\end{equation}
\end{lemma}

\begin{proof}
In the proof of Lemma \ref{lem:T-parity} we have shown that
\begin{equation}
T_{\rm even}
=2\int_0^\infty\frac{r}{e^{\pi r}-1}
\,a_r\otimes a_r\dd r
\qquad\text{with}\
a_r(\theta):=\cos\theta\cosh(r\theta) \,.
\end{equation}
Taking the trace (which is well-defined for any nonnegative operator) and using Tonelli's
theorem, together with
$2\cosh^2(r\theta)=1+\cosh(2r\theta)$ and
\eqref{eq:fdoubleprimerepr}, gives
\begin{align*}
\tau
&=2\int_0^\infty\frac{r}{e^{\pi r}-1}\norm{a_r}_{L^2(I)}^2\dd r
=\frac12\int_I\cos^2\theta
\bigl(F''(0)+F''(2\theta)\bigr)\dd\theta \\
& =\frac{\pi}{12}
+\int_0^{\pi/2}\cos^2\theta F''(2\theta)\dd\theta 
=\frac{\pi}{12}
+\frac14\int_0^\pi(1+\cos t)F''(t)\dd t.
\end{align*}
Since $F'(t)=t^{-1}-\cot t$, integration by parts yields
\begin{align*}
\int_0^\pi(1+\cos t)F''(t)\dd t
&=\int_0^\pi\sin t\,F'(t)\dd t
=\int_0^\pi\left(\frac{\sin t}{t}-\cos t\right)\dd t
=\operatorname{Si}(\pi).
\end{align*}
Here the boundary term vanishes at both endpoints.  This proves
\eqref{eq:tau-exact}.

We next compute the Rayleigh quotient.  By symmetry and the change of
variables $t=\theta-\varphi$,
\begin{equation}\label{eq:a0-convolution}
a_0=\frac4\pi\int_0^\pi W(t)F''(t)\dd t,
\end{equation}
where
\[
W(t):=\int_{-\pi/2+t}^{\pi/2}
\cos^2\theta\cos^2(\theta-t)\dd\theta =\frac{\pi-t}{4}
+\frac{\pi-t}{8}\cos(2t)+\frac{3}{16}\sin(2t) \,.
\]
Integrating by parts in \eqref{eq:a0-convolution}, using
\begin{equation}\label{eq:W-prime}
W'(t)=\frac{t-\pi}{4}\sin(2t)
+\frac{\cos(2t)-1}{4}
\end{equation}
and $F'(t)=t^{-1}-\cot t$ and noting that the boundary terms vanish, gives
\begin{align*}
\int_0^\pi W(t)F''(t)\dd t
&=-\int_0^\pi W'(t)F'(t)\dd t
=\frac14\int_0^\pi
\left((\pi-t)\sin(2t)+1-\cos(2t)\right)
\left(\frac1t-\cot t\right)\dd t.
\end{align*}
The part containing $1/t$ equals
\begin{align*}
&\int_0^\pi
\frac{(\pi-t)\sin(2t)+1-\cos(2t)}{t}\dd t =\pi\operatorname{Si}(2\pi)
+\operatorname{Cin}(2\pi),
\end{align*}
whereas the part containing $\cot t$ equals
\begin{align*}
&\int_0^\pi
\left((\pi-t)\sin(2t)+1-\cos(2t)\right)\cot t\dd t
=\int_0^\pi
\left(2(\pi-t)\cos^2t+2\sin t\cos t\right)\dd t
=\frac{\pi^2}{2}.
\end{align*}
Together with \eqref{eq:a0-convolution}, this proves
\eqref{eq:a0-exact}.

The numerical inequalities in
\eqref{eq:tau-a0-numerical} follow from the fact that the Taylor series
\begin{align*}
\operatorname{Si}(x)
=\sum_{k=0}^\infty
\frac{(-1)^k x^{2k+1}}{(2k+1)(2k+1)!},
\qquad
\operatorname{Cin}(x)
=\sum_{k=1}^\infty
\frac{(-1)^{k+1}x^{2k}}{2k(2k)!}.
\end{align*}
are alternating. For $0<x\leq2\pi$, the absolute values of the terms in both tails
are decreasing from the third nonzero term onward.  Using
$3.1415926<\pi<3.1415927$, summing the series for
$\operatorname{Si}(\pi)$ through $k=8$ and those for
$\operatorname{Si}(2\pi)$ and $\operatorname{Cin}(2\pi)$ through
$k=13$, and bounding each remaining tail by the next term gives
\begin{align*}
1.8519370<\operatorname{Si}(\pi)&<1.8519380,\\
1.4181515<\operatorname{Si}(2\pi)&<1.4181517,\\
2.4376533<\operatorname{Cin}(2\pi)&<2.4376535.
\end{align*}
Substitution in \eqref{eq:tau-exact} and \eqref{eq:a0-exact} gives
\eqref{eq:tau-a0-numerical}.
\end{proof}


\subsection*{The residual of the trial vector}

We now show that $e_0$ is an approximate eigenvector of $T_{\rm even}$. Let
\[
\Delta_n:=b_n-b_{n+1} \,,
\qquad\text{where}\ 
b_1:=\frac12,
\qquad
b_n:=\frac{\ln n}{n^2-1}\quad(n\geq2) \,.
\]

\begin{lemma}\label{lem:e0-residual}
With $a_0$ as in Lemma~\ref{lem:trace-and-rayleigh}, one has
\begin{equation}\label{eq:rho-explicit-bound}
\rho_0^2:=\norm{T_{\rm even}e_0-a_0e_0}^2 \leq \frac{1}{\pi^2} \left(\Delta_1^2+\Delta_2^2+\Delta_3b_3\right).
\end{equation}
In particular,
\begin{align}
    \label{eq:e0-residual}
    \rho_0^2 < 0.0089.
\end{align}
\end{lemma}

\begin{proof}
For $n\geq0$, put
\begin{equation}\label{eq:en-def}
e_n(\theta):=\sqrt{\frac2\pi}\cos((2n+1)\theta) \,.
\end{equation}
The functions $(e_n)_{n\geq0}$ form an orthonormal basis of the even
subspace of $L^2(I)$. Using
$2\cos\theta\cos((2n+1)\theta)
=\cos(2n\theta)+\cos(2(n+1)\theta)$ and integrating each term, we find
\begin{equation}\label{eq:An-formula}
\int_I e_n(\theta)\cos\theta\cosh(r\theta)\dd\theta
=(-1)^n\sqrt{\frac2\pi}\,
 r\sinh\left(\frac{\pi r}{2}\right)d_n(r)
\end{equation}
with
\begin{equation}\label{eq:dn-def}
d_n(r):=\frac{1}{r^2+4n^2}
-\frac{1}{r^2+4(n+1)^2} \,.
\end{equation}

Combining \eqref{eq:An-formula} with
\eqref{eq:fdoubleprimerepr},
we obtain
\begin{equation}\label{eq:T-matrix-e0}
\ip{e_n}{T_{\rm even}e_0}
=\frac{(-1)^n}{\pi}\int_0^\infty
r^3(1-e^{-\pi r})d_n(r)d_0(r)\dd r.
\end{equation}
For $n\geq1$, all factors in the integral
are nonnegative.  Dropping the factor $1-e^{-\pi r}$ gives
\begin{equation}\label{eq:T-matrix-upper}
\abs{\ip{e_n}{T_{\rm even}e_0}}
\leq\frac1\pi\int_0^\infty r^3d_n(r)d_0(r)\dd r = \frac1\pi \left( b_n-b_{n+1}\right) = \frac1\pi\,\Delta_n \,,
\end{equation}
where the first equality comes from
\[
4\int_0^\infty
\frac{r\dd r}{(r^2+4)(r^2+4a^2)}
=\frac{\ln a}{a^2-1}\quad(a>1),
\]
with the limiting value $1/2$ at $a=1$.
The function
\[
b(x):=\frac{\ln x}{x^2-1}
\]
is decreasing and convex on $[3,\infty)$.  Indeed,
\[
b'(x)=\frac{x^2-1-2x^2\ln x}{x(x^2-1)^2}<0
\]
and
\[
b''(x)=
\frac{6x^4\ln x-5x^4+2x^2\ln x+6x^2-1}
{x^2(x^2-1)^3}>0
\qquad(x\geq3).
\]
For $x\geq3$ one has $\ln x>1$, so the numerator of $b'(x)$
is smaller than $-x^2-1$, while the numerator of $b''(x)$ is larger
than $x^4-1$.  Thus the displayed signs are justified, and it follows
that $(\Delta_n)_{n\geq3}$ is decreasing.  Hence, by
Parseval's identity, \eqref{eq:T-matrix-upper}, and telescoping,
\begin{align*}
\rho_0^2
&=\sum_{n=1}^\infty
\abs{\ip{e_n}{T_{\rm even}e_0}}^2\notag
\leq\frac1{\pi^2}\sum_{n=1}^\infty\Delta_n^2\notag
\leq\frac1{\pi^2}
\left(\Delta_1^2+\Delta_2^2
+\Delta_3\sum_{n=3}^\infty\Delta_n\right)\notag
=\frac1{\pi^2}
\left(\Delta_1^2+\Delta_2^2+\Delta_3b_3\right),
\end{align*}
proving \eqref{eq:rho-explicit-bound}.
Finally, to prove the numerical bound \eqref{eq:e0-residual}, we use
\[
3.14159<\pi,
\qquad
0.693147<\ln2<0.693148,
\qquad
1.098612<\ln3<1.098613 \,,
\]
which imply
\[
\Delta_1<0.268951,
\quad
\Delta_2<0.093723,
\quad
\Delta_3<0.044908,
\quad
b_3<0.137327.
\]
Substitution into \eqref{eq:rho-explicit-bound} gives
\[
\rho_0^2
<\frac{0.268951^2+0.093723^2+0.044908\cdot0.137327}
{3.14159^2}
<0.00885<0.0089,
\]
which proves the lemma.
\end{proof}


\subsection*{Proof of Proposition \ref{prop:normbound}}

We apply Lemma~\ref{lem:trace-residual} to
$A=T_{\rm even}$ and $v=e_0$.  Lemmas
\ref{lem:trace-and-rayleigh} and \ref{lem:e0-residual} give
\[
\tau<0.725,
\qquad
a_0<0.6234,
\qquad
\rho_0^2<0.0089.
\]
For
\[
G(\tau,a,\eta)
:=\frac{\tau+\sqrt{(2a-\tau)^2+4\eta}}{2},
\]
one has $\partial_\tau G>0$, $\partial_\eta G>0$, and
\[
\partial_aG
=\frac{2a-\tau}{\sqrt{(2a-\tau)^2+4\eta}}>0
\]
in the present numerical range.  Therefore
\begin{align*}
\norm{T_{\rm even}}
&<\frac{0.725+
\sqrt{(2\cdot0.6234-0.725)^2+4\cdot0.0089}}{2}<0.639933<0.64.
\end{align*}
This completes the proof of Proposition \ref{prop:normbound}.
\qed


\section{The maximal Kato strip of the arctangent}\label{sec:arctank1}

Recall from \eqref{eq:theta-s} that $\vartheta_s(x)=\arctan(x/s)$.

\begin{proposition}\label{prop:arctan-strip}
For $s>0$, the function $\vartheta_s$ belongs to $K_s$.  It does not belong to $K_a$ for any $a>s$. The function $-\vartheta_s$ does not belong to $K_a$ for any $a>0$.
\end{proposition}

\begin{proof}
It is enough to prove the assertion for $s=1$, since $\vartheta_s(z)=\vartheta_1(z/s)$.

The strip $S_1$ is simply connected and does not contain either of the two zeros $\pm i$ of $1+z^2$.  Therefore the function
\[
\Theta(z):=\int_0^z\frac{\dd w}{1+w^2},
\qquad z\in S_1,
\]
is well defined and analytic.  On the real axis it agrees with the usual real arctangent $\vartheta_1$.

For $z=x+iy\in S_1$, the logarithmic formula
\[
\Theta(z)=\frac{1}{2i}\log\frac{1+iz}{1-iz}
\]
can be taken with branches that are analytic in the strip.  Taking the imaginary part gives
\begin{equation}\label{eq:Im-arctan}
\operatorname{Im}\Theta(x+iy)
=\frac14\ln
\frac{x^2+(1+y)^2}{x^2+(1-y)^2}.
\end{equation}
If $y>0$, then $(1+y)^2>(1-y)^2$, so the right side is positive.  If $y<0$, it is negative.  Thus
\[
\operatorname{Im}\Theta(z)\operatorname{Im}z\geq0,
\]
and hence $\vartheta_1\in K_1$.

Suppose now that $\vartheta_1$ belonged to $K_a$ for some $a>1$, and let $G$ be its analytic continuation to $S_a$.  On the real axis,
\[
G'(x)=\frac{1}{1+x^2}.
\]
The function
\[
H(z):=(1+z^2)G'(z)-1
\]
is analytic throughout $S_a$ and vanishes on the real axis.  By the identity theorem, $H$ vanishes identically.  Since $i\in S_a$, evaluation at $z=i$ gives
\[
0=H(i)=-1,
\]
a contradiction.  Thus the maximal strip half-width for $\vartheta_1$ is one.

Similarly, suppose that $-\vartheta_1$ belonged to $K_a$ for some $a>0$ and let $\tilde G$ be its analytic continuation to $S_a$. With $\Theta$ as before, we see that the analytic function $\tilde G+\Theta$ vanishes on the real axis and therefore in $S_{\min\{1,a\}}$. Thus, for all $z\in S_{\min\{1,a\}}$,
$$
\operatorname{Im} \tilde G(z)\operatorname{Im} z = -\operatorname{Im} \Theta(z)\operatorname{Im} z \leq 0 \,,
$$
Together with the assumption $\tilde G\in K_a$, we obtain $\operatorname{Im} \tilde G(z)=0$ for all $z\in S_{\min\{1,a\}}\setminus\R$. By a well-known theorem about analytic functions, this implies that $\tilde G$ vanishes identically, a contradiction.
\end{proof}


\section*{Acknowledgments}
R.L.F.~acknowledges partial support through the German Research Foundation grants EXC-2111-390814868 and TRR 352--Project-ID 470903074. P.I.~acknowledges partial support from the US NSF CAREER grant DMS-2152401, US NSF grant DMS-2554183, a Simons Fellowship, and a Humboldt Research Fellowship for Experienced Researchers. The authors acknowledge the use of AI tools. All mathematical arguments and proofs in the final manuscript were checked and written by the authors.


\end{document}